\documentclass[11pt]{amsart}
\usepackage{mathrsfs,latexsym,amsfonts,amssymb}
\newtheorem{theorem}{Theorem}[section]
\newtheorem{lemma}[theorem]{Lemma}
\newtheorem{corollary}[theorem]{Corollary}
\newtheorem{question}[theorem]{Question}

\theoremstyle{definition}
\newtheorem{definition}[theorem]{Definition}
\newtheorem{proposition}[theorem]{Proposition}
\newtheorem{example}[theorem]{Example}
\newtheorem{problem}[theorem]{Problem}

\begin{document}

\title[Some generalized metric properties on hyperspaces with the Vietoris topology]
{Some generalized metric properties on hyperspaces with the Vietoris topology}

 \author{Fucai Lin}
  \address{(Fucai Lin): School of mathematics and statistics,
  Minnan Normal University, Zhangzhou 363000, P. R. China}
  \email{linfucai2008@aliyun.com; linfucai@mnnu.edu.cn}

\author{Rongxin Shen}
\address{(Rongxin Shen): Department of Mathematics, Taizhou University, Taizhou 225300, P. R. China} \email{srx20212021@163.com}

\author{Chuan Liu}
\address{(Chuan Liu): Department of Mathematics,
Ohio University Zanesville Campus, Zanesville, OH 43701, USA}
\email{liuc1@ohio.edu}

  \thanks{The first author is supported by the NSFC (No. 11571158), the Natural Science Foundation of Fujian Province (No. 2017J01405) of China, the Program for New Century Excellent Talents in Fujian Province University, the Institute of Meteorological Big Data-Digital Fujian and Fujian Key Laboratory of Date Science and Statistics.}

\keywords{hyperspace; generalized metric space; $k$-network; point-countable base; point-regular base; $\aleph$-space; c-semi-stratifiable space; g-first-countable space.}
\subjclass[2010]{54B20; 54D20}

\begin{abstract}
We study the heredity of the classes of generalized metric spaces (for example, spaces with a $\sigma$-hereditarily closure-preserving $k$-network, spaces with a point-countable base, spaces with a base of countable order, spaces with a point-regular base, Nagata-spaces, $c$-semi-stratifiable spaces, $\gamma$-spaces, semi-metrizable spaces) to the hyperspaces of nonempty
compact subsets and finite subsets with the Vietoris topology.
\end{abstract}

\maketitle
\section{Introduction}
It is well known that the topics of the hyperspace has been the focus of much research, see \cite{GM2016, K1978, M1951, NP2015, N1985, PS2017, TLL2018}. There are many results on the hyperspace $CL(X)$ of closed subsets of a topological space equipped with various topologies and various of its subsets such as $\mathcal{F}(X)$, the space of finite subsets of $X$, and $\mathcal{K}(X)$, the space of compact subsets of $X$. In this paper, we endow
$\mathcal{K}(X)$ with the Vietoris topology, or the so-called finite topology, the base of which consists
of all subsets of the following form:
$$\langle U_1, ..., U_k\rangle =\{K\in \mathcal{K}(X): K\subset
{\bigcup}_{i=1}^k U_i\ \mbox{and}\ K\cap U_j\neq \emptyset, 1\leq j\leq
k\},$$where each $U_i$ is open in $X$ and $k\in\mathbb{N}$. From the relationship of topological properties between $X$ and $\mathcal{F}(X)$ and $\mathcal{K}(X)$, we naturally consider the following problem:

\begin{problem}
Let $\mathcal{C}$ be a class of spaces with some generalized metric property. If $X\in\mathcal{C}$, does then $\mathcal{K}(X)$ or $\mathcal{F}(X)$ belong to $\mathcal{C}$?
\end{problem}

The paper is organized as follows. In Section 2, we
introduce the necessary notation and terminology which are used for
the rest of the paper. In Section 3, we mainly discuss the generalized metric properties on $\mathcal{K}(X)$, such as, $\aleph$-spaces, spaces with a point-countable base and $\alpha$-spaces, etc. In Section 4, we mainly discuss the generalized metric properties on $\mathcal{F}(X)$, such as, spaces with a BCO, spaces with a sharp base, Nagata spaces, spaces with a point-regular base, $c$-semi-stratifiable spaces, $\gamma$-spaces, $g$-first-countable spaces, etc.

\maketitle
\section{Preliminaries}
As it is known already from \cite{M1951}, most of separation axioms of a space $X$, including
regularity and $T_{2}$-ness, are inherited to $\mathcal{K}(X)$. Therefore, throughout this paper, all topological spaces are assumed to be
at least regular $T_{2}$-space, unless explicitly stated otherwise. Let $\mathbb{N}$ and $\omega$ denote the sets of all positive
  integers and all non-negative integers, respectively. For undefined
  notation and terminologies, the reader may refer to
  \cite{E1989} and \cite{G1984}.

  Let $X$ be a topological space and $A \subseteq X$ be a subset of $X$.
  The \emph{closure} of $A$ in $X$ is denoted by $\overline{A}$. A subset $P$ of $X$ is called a
  \emph{sequential neighborhood} of $x \in X$, if each
  sequence converging to $x$ is eventually in $P$. A subset $U$ of
  $X$ is called \emph{sequentially open} if $U$ is a sequential neighborhood of
  each of its points. A subset $F$ of
  $X$ is called \emph{sequentially closed} if $X\setminus F$ is sequentially open. The space $X$ is called a \emph{sequential space} if each
  sequentially open subset of $X$ is open. The space $X$ is said to be {\it Fr\'{e}chet-Urysohn} if, for
each $x\in \overline{A}\subset X$, there exists a sequence
$\{x_{n}\}$ in $A$ such that $\{x_{n}\}$ converges to $x$.

\begin{definition}
Let $\mathscr P$ be a cover of a space $X$ such that (i) $\mathscr P =\bigcup_{x\in X}\mathscr{P}_{x}$; (ii) for each point $x\in X$, if $U,V\in\mathscr{P}_{x}$, then $W\subseteq U\cap V$ for some $W\in \mathscr{P}_{x}$;
and (iii) for each point $x\in X$ and each open neighborhood $U$ of $x$ there is some $P\in\mathscr P_x$ such that $x\in P \subseteq U$.

\smallskip
$\bullet$ The family $\mathscr P$ is called an \emph{sn-network} for $X$ if for each point $x\in X$, each element of $\mathscr P_x$ is a sequential neighborhood of $x$ in $X$, and $X$ is called \emph{snf-countable} if $X$ has an $sn$-network $\mathscr P$ and $\mathscr P_x$ is countable for all $x\in X$.

\smallskip
$\bullet$ The family $\mathscr{P}$ is called a {\it weak base} for $X$ if, for every $G\subset X$, the set $G$ must be open in $X$ whenever for each $x\in G$ there exists $P\in
\mathscr{P}_{x}$ such that $P\subset G$, and $X$ is {\it g-first-countable} if $X$ has a weak base $\mathscr P$ and $\mathscr{P}_{x}$ is countable for each $x\in X$.
\end{definition}

The following
  implications follow directly from definitions:
  \[
  \mbox{first countable space} \Rightarrow \mbox{g-first-countable space}\Rightarrow
\mbox{$snf$-countable space, and}\footnote{ Note that none of these implications can be reversed.}
  \]
  \[
  \mbox{sequential space + $snf$-countable space} \Leftrightarrow \mbox{g-first-countable space}.
  \]

  \begin{definition}
  Let $\lambda$ be a cardinal. An indexed family $\{X_{\alpha}\}_{\alpha\in\lambda}$ of subsets of a space $X$ is called

\smallskip
 $\bullet$ {\it point-countable} if for any point $x\in X$ the set $\{\alpha\in\lambda: x\in X_{\alpha}\}$ is countable;

\smallskip
 $\bullet$ {\it compact-countable} if for any compact subset $K$ in $X$ the set $\{\alpha\in\lambda: K\cap X_{\alpha}\neq\emptyset\}$ is countable;

\smallskip
  $\bullet$ {\it locally countable} if any point $x\in X$ has a neighborhood $O_{x}\subset X$ such that the set $\{\alpha\in\lambda: O_{x}\cap X_{\alpha}\neq\emptyset\}$ is countable;

\smallskip
  $\bullet$ {\it locally finite} if any point $x\in X$ has a neighborhood $O_{x}\subset X$ such that the set $\{\alpha\in\lambda: O_{x}\cap X_{\alpha}\neq\emptyset\}$ is finite;

\smallskip
  $\bullet$ {\it compact-finite} in $X$ if for each compact subset $K\subset X$ the set $\{\alpha\in\lambda: K\cap X_{\alpha}\neq\emptyset\}$ is finite.
  \end{definition}

\begin{definition}
Let $\mathscr P$ be a family of subsets of a space $X$.

\smallskip
$\bullet$ The family $\mathscr{P}$ is
{\it hereditarily closure-preserving} (abbrev. HCP) if, whenever a
subset $S(P)\subset P$ is chosen for each $P\in\mathscr{P}$, the
family $\{S(P): P\in \mathscr{P}\}$ is closure-preserving.

\smallskip
$\bullet$ The family $\mathscr P$ is called a {\it $k$-network}
if for every compact subset $K$ of $X$ and an arbitrary open set
$U$ containing $K$ in $X$ there is a finite subfamily $\mathscr {P}^{\prime}\subseteq\mathscr {P}$ such that $K\subseteq \bigcup\mathscr {P}^{\prime}\subseteq U$.

\smallskip
$\bullet$ The space $X$ is an \emph{$\aleph$-space} if
$X$ has a $\sigma$-locally finite $k$-network.

\smallskip
$\bullet$ The space $X$ is said to be
\emph{La\v{s}nev} if it is the continuous closed image of some metric space.
\end{definition}

The following La\v{s}nev spaces in Definition~\ref{d00} play an important role in the study of the generalized metric theory. Moreover, if $\kappa$ be a uncountable cardinal, then it is not an $\aleph$-space.

\begin{definition}\label{d00}
Let $\kappa$ be an infinite cardinal.
  For each $\alpha\in\kappa$, let $T_{\alpha}$ be a sequence converging to
  $x_{\alpha}\not\in T_{\alpha}$. Let $T=\bigoplus_{\alpha\in\kappa}(T_{\alpha}\cup\{x_{\alpha}\})$ be the topological sum of $\{T_{\alpha}
  \cup \{x_{\alpha}\}: \alpha\in\kappa\}$. Then
  $S_{\kappa}=\{x\}  \cup \bigcup_{\alpha\in\kappa}T_{\alpha}$
  is the quotient space obtained from $T$ by
  identifying all the points $x_{\alpha}\in T$ to the point $x$. The space $S_{\kappa}$ is called a {\it sequential fan}.
\end{definition}

The following space in Definition~\ref{d01} is an $\aleph$-space which is not a La\v{s}nev space.

\begin{definition}\label{d01}
A space $X$ is called an \emph{ $S_{2}$}-{space} ({\it Arens' space})  if
$$X=\{\infty\}\cup \{x_{n}: n\in \mathbb{N}\}\cup\{x_{n, m}: m, n\in
\omega\}$$ and the topology is defined as follows: Each
$x_{n, m}$ is isolated; a basic neighborhood of $x_{n}$ is
$\{x_{n}\}\cup\{x_{n, m}: m>k\}$, where $k\in\omega$;
a basic neighborhood of $\infty$ is $$\{\infty\}\cup (\bigcup\{V_{n}:
n>k\})\ \mbox{for some}\ k\in \omega,$$ where $V_{n}$ is a
neighborhood of $x_{n}$ for each $n\in\omega$.
\end{definition}

Given a space $X$, we define its {\it hyperspace} as the
following sets.
$$\mathcal{K}(X)=\{K\subset X, K\ \mbox{is compact}\},$$
$$\mathcal{F}_n(X)=\{A\subset X, |A|\leq n\}, \mbox{where}\ n\in\mathbb{N}, \mbox{and}$$
$$\mathcal{F}(X)=\{A\subset X, A\ \mbox{is finite}\}.$$
We endow $\mathcal{K}(X)$ with the {\it Vietoris topology}. Obviously, from the definitions above, we have that $\mathcal{F}(X),
\mathcal{F}_n(X)$ are subspaces of $\mathcal{K}(X)$ and
$\mathcal{F}(X)={\bigcup_{n=1}^{\infty}}\mathcal{F}_n(X)$.

\smallskip
\section{Generalized metric properties on $\mathcal{K}(X)$}
In this section, we mainly discuss the generalized metric properties on $\mathcal{K}(X)$, and study the relations of the topological property $\mathcal{P}$ in $X$ and $\mathcal{K}(X)$. First, we recall two lemmas in \cite{M1951}.

Let $\mathcal{P}$ be a family of subsets of a space $X$. For any
$r\in\mathbb{N}$ and $P_1, ..., P_r\in \mathcal{P}$, denote $\langle
P_1, ..., P_r\rangle$ by the set $$\{K\in \mathcal{K}(X): K\subset
{\bigcup_{i=1}^r}P_i\ \mbox{and}\ K\cap P_j\neq \emptyset, 1\leq
j\leq r\}.$$ Put $$\mathcal{B}=\{\langle P_1, ..., P_r\rangle:
P_i\in\mathcal{P}, 1\leq i\leq r, r\in \mathbb{N}\}.$$

\begin{lemma}\label{l1}
\cite[Lemma 2.3.1]{M1951} Let $U_1, ..., U_n$ and $V_1,..., V_m$ be subsets of a space $X$. Then $\langle U_1, ..., U_n\rangle\subset
\langle V_1,..., V_m\rangle$ if and only if
$\bigcup_{j=1}^nU_j\subset \bigcup_{j=1}^mV_j$ and for every $V_i$
there exists a $U_k$ such that $U_k\subset V_i$.
\end{lemma}

\begin{lemma}\label{l2}
\cite[Lemma 2.3.2]{M1951} Let $U_1, ..., U_n$ be subsets of a space $X$. Then $$Cl(\langle U_1, ..., U_n\rangle)=\langle
\overline{U}_1, ..., \overline{U}_n\rangle.$$
\end{lemma}

The following Lemma~\ref{l3} maybe was proved somewhere.

\begin{lemma}\label{l3}
Let $\mathcal{P}$ be a $k$-network of space $X$. Then $\mathcal{B}$
is a $k$-network of $\mathcal{K}(X)$.
\end{lemma}

\begin{proof}
Let $\mathbf{K}\subset \widehat{U}$ with $\mathbf{K}$ compact and
$\widehat{U}$ open in $\mathcal{K}(X)$. We divide the proof into the
following two cases.

\smallskip
{\bf Case 1.} $\widehat{U}$ is a basic open subset of $\mathcal{K}(X)$.

\smallskip
We write $\widehat{U}=\langle U_1, ..., U_k\rangle$, where each $U_i
(i\leq k)$ is open in $X$. Since $\mathbf{K}$ is compact in
$\mathcal{K}(X)$ and $\mathcal{K}(X)$ is regular, there exist
finitely many basic open subsets $\{\widehat{V}(i): i\leq n\}$ such that
$$\mathbf{K}\subset \bigcup_{i\leq n}\widehat{V}(i)\subset
\bigcup_{i\leq n}Cl(\widehat{V}(i))\subset \widehat{U},$$ where each
$\widehat{V}(i)=\langle V_1(i), ..., V_{s_i}(i)\rangle$ and each $V_j(i)
(j\leq s_i, i\leq n)$ is open in $X$.

Fix an arbitrary $i\leq n$. Since $$Cl(\widehat{V}(i))=\langle
\overline{V_1(i)}, ..., \overline{V_{s_i}(i)}\rangle\subset \widehat{U}$$
by Lemma~\ref{l2}, it follows from Lemma~\ref{l1} that for each $t\leq k$ we can take
$i(t)\leq s_i$ such that $\overline{V_{i(t)}(i)}\subset U_t$. Let
$$A=\{i(t): t\leq k\}, B=\{1, \ldots, s_i\}\setminus A.$$ By
\cite[Theorem 0.2]{N1985}, we see that $$K(i)=\bigcup\{K: K\in
\mathbf{K}\cap Cl(\widehat{V}(i))\}$$ is compact in $X$, hence it is
easy that $$K(i)\in Cl(\widehat{V}(i))=\langle \overline{V_1(i)}, ...,
\overline{V_{s_i}(i)}\rangle\subset \langle U_1, ..., U_k\rangle.$$
For each $j\in A$, since $\mathcal{P}$ is a $k$-network of
$X$, we can choose a finite subfamily $\mathcal{P}_j\subset
\mathcal{P}$ such that $(K(i)\cap\overline{V_j(i)})\cap P\neq
\emptyset$ for each $P\in \mathcal{P}_j$ and $$K(i)\cap
\overline{V_j(i)}\subset \bigcup\mathcal{P}_j\subset
\bigcap\{U_t: t\leq k, i(t)=j\}.$$ For each $j\in B$, choose a finite subfamily
$\mathcal{P}_j\subset \mathcal{P}$ such that
$(K(i)\cap\overline{V_j(i)})\cap P\neq \emptyset$ for $P\in
\mathcal{P}_j$ and $K(i)\cap \overline{V_j(i)}\subset
\bigcup\mathcal{P}_j\subset \bigcup_{t=1}^{k}U_t$. Let
$$\mathscr{W}_{i}=\{\bigcup_{j=1}^{s_i}\mathcal{P}_j': \ \mbox{for
any nonempty subfamily}\ \mathcal{P}_j'\subset \mathcal{P}_j, j\leq
s_i\}.$$ Obviously, $\mathscr{W}_{i}$ is a finite set consisting of
elements of some finite subfamilies of $\mathcal{P}$. Then
$$\mathbf{K}\cap Cl(\widehat{V}(i))\subset \bigcup\{\langle
\mathcal{W}\rangle: \mathcal{W} \in \mathscr{W}_{i}\}\subset
\widehat{U},$$ where $\langle \mathcal{W}\rangle=\langle P_1,...,
P_m\rangle$ for $\mathcal{W}=\{P_1,..., P_m\}\in\mathscr{W}_{i}$.

Now let $\mathscr{W}=\bigcup_{i=1}^{n}\mathscr{W}_{i}$ and let
$\mathscr{O}=\{<\mathcal{W}>: \mathcal{W}\in\mathscr{W}\}$.
Obviously, $\mathscr{O}$ is a finite subset of $\mathcal{B}$. Of
course, we have $\mathbf{K}\subset
\bigcup\mathscr{O}\subset\widehat{U}$.

\smallskip
{\bf Case 2}: $\widehat{U}$ is an open subset of $\mathcal{K}(X)$

\smallskip
By the compactness of $\mathbf{K}$, there exists a
finite open cover $\{\widehat{U}(i): i=1, 2, ..., p\}$ of $\mathbf{K}$
consisting of basic open subsets of $\mathcal{K}(X)$ such that
$\mathbf{K}\subset \bigcup_{i=1}^{p}\widehat{U}(i)\subset \widehat{U}$. Since
$\mathbf{K}$ is compact, there exists a closed cover
$\{\mathbf{K}(i): i=1, 2, ...,p\}$ of $\mathbf{K}$ such that
$\mathbf{K}(i)\subset \widehat{U}(i)$ for each $i\leq p$. By virtue of
Case 1, for each $i\leq p$ we can find a finite subset
$\mathscr{O}_{i}$ of $\mathcal{B}$ such that $\mathbf{K}(i)\subset
\bigcup\mathscr{O}_{i}\subset \widehat{U}(i)$, hence $\mathbf{K}\subset
\bigcup_{i=1}^p\mathscr{O}_{i}\subset \widehat{U}$.
\end{proof}

\begin{proposition}\label{l4}
Let $\mathcal{P}$ be a compact-countable (resp., locally countable,
locally finite, compact-finite) family of subsets of $X$. Then $\mathcal{B}$ is
compact-countable (resp., locally countable, locally finite, compact-finite) in
$\mathcal{K}(X)$. Moreover, if $\mathcal{P}$ is a family of closed subsets, then $\mathcal{B}$ is a family of closed subsets in $\mathcal{K}(X)$.
\end{proposition}

\begin{proof}
We consider the following four cases respectively.

\smallskip
1) The family $\mathcal{P}$ is compact-countable.

\smallskip
Take any compact subset $\mathbf{K}$ of $\mathcal{K}(X)$. Put
$K=\bigcup\mathbf{K}$. Then $K$ is compact in $X$ by \cite[Theorem
2.5.2]{M1951}. Put $\mathcal{P}'=\{P\in \mathcal{P}: P\cap K\neq
\emptyset\}$. Then $\mathcal{P}'$ is countable since $\mathcal{P}$
is compact-countable in $X$. Take any $\langle P_1,..,
P_s\rangle\in\mathcal{B}$. Assume that $\mathbf{K}\cap \langle
P_1,.., P_s\rangle\neq \emptyset$, then we can choose an arbitrary $H\in \mathbf{K}\cap
\langle P_1,.., P_s\rangle$. Then for each $i\leq s$ we have
$P_i\cap H\neq\emptyset$,  hence $P_i\cap K\neq \emptyset$ since
$H\subset K$, which shows that $P_i\in \mathcal{P}'$ for each $i\leq
p$. Therefore, $\mathcal{B}$ is compact-countable.

\smallskip
2) The family $\mathcal{P}$ is locally countable.

\smallskip
Take any point $K$ of $\mathcal{K}(X)$. We need to find a
neighborhood $\widehat{U}$ of $K$ in $\mathcal{K}(X)$ such that
$\widehat{U}$ intersects at most countably many elements of
$\mathcal{B}$. For each $x\in K$, there is an open neighborhood
$U_x$ of $x$ such that $U_x$ intersects at most countably many elements of $\mathcal{P}$. Then $\{U_x: x\in K\}$ is a cover of $K$, hence there is a finite subset $\{x_{i}: i\leq n\}$
of $K$ such that $K\subset \{U_{x_i}: i\leq n\}$. Let $U=\bigcup_{i=1}^{n}U_{x_i}$. Clearly,
$|\{P\in \mathcal{P}: P\cap U\neq \emptyset\}|\leq \omega$. Let
$\widehat{U}=\langle U_{x_1}, .., U_{x_n}\rangle$. Take any $\langle
P_1,.., P_s\rangle\in\mathcal{B}$. If $\widehat{U}\cap \langle P_1,...,
P_s\rangle\neq \emptyset$, then we can pick any $H\in \langle U_{x_1}, ...,
U_{x_n}\rangle\cap \langle P_1,..., P_s\rangle$, hence $H\subset U,
H\cap U_{x_{i}}\neq \emptyset, H\cap P_j\neq \emptyset$ for each
$i\leq n, j\leq s$. Thus $U\cap P_j\neq \emptyset$ for each $j\leq
s$, which shows that each $P_j\in \{P\in\mathcal{P}: P\cap U\neq \emptyset\}$.
Therefore, $\mathcal{B}$ is locally countable.

\smallskip
3) The family $\mathcal{P}$ is locally finite.
\smallskip

The proof is similar to 2), so we left it to the reader.

\smallskip
4) The family $\mathcal{P}$ is compact-finite.
\smallskip

The proof is similar to 1), so we left it to the reader.

By Lemma~\ref{l2}, it is obvious that $\mathcal{B}$ is a family of closed subsets in $\mathcal{K}(X)$ if $\mathcal{P}$ is a family of closed subsets.
\end{proof}

From Lemma~\ref{l3} and Proposition~\ref{l4}, we have the following two corollaries.

\begin{corollary}
A space is an $\aleph_{k}$-space\footnote{Recall that a regular space $X$ is $\aleph_{k}$ ($\overline{\aleph}_{k}$) \cite{B2016} if $X$ has a compact-countable (closed) $k$-network.} ($\overline{\aleph}_{k}$-space) if and only if $\mathcal{K}(X)$ is an $\aleph_{k}$-space ($\overline{\aleph}_{k}$-space).
\end{corollary}

\begin{corollary}\cite{K1978}\label{cc0}
A space is an $\aleph$-space if and only if $\mathcal{K}(X)$ is an $\aleph$-space.
\end{corollary}

Indeed, we prove that, for space $X$, if $\mathcal{F}_2(X)$ has a $\sigma$-HCP $k$-network then $\mathcal{K}(X)$ is an $\aleph$-spac, see Theorem~\ref{tt2}. Foged introduced $\sigma$-hereditarily closure-preserving (abbr.
HCP) $k$-network to characterize the closed image of a metric space,
every $\aleph$-space has a $\sigma$-HCP $k$-network; however, a
space with a $\sigma$-HCP $k$-network need not be an $\aleph$-space, such as $S_{\omega_{1}}$. But the following theorem shows that if $\mathcal{F}_{2}(X)$ has a $\sigma$-HCP $k$-network then $\mathcal{K}(X)$ is an $\aleph$-space.

\begin{theorem}\label{tt2}
The following are equivalent for a space $X$.
\begin{enumerate}
\item $X$ is an $\aleph$-space;
\item $\mathcal{F}_2(X)$ has a $\sigma$-HCP $k$-network;
\item $\mathcal{F}(X)$ has a $\sigma$-HCP $k$-network;
\item $\mathcal{F}(X)$ is an $\aleph$-space.
\item $\mathcal{K}(X)$ is an $\aleph$-space.
\end{enumerate}
\end{theorem}

\begin{proof}
(5) $\to$ (4) $\to$ (3) $\to$ (2) are trivial. By Corollary~\ref{cc0}, we have (1) $\to$ (5). It suffices to prove (2) $\to$ (1).

Let $\mathcal{P}=\bigcup_{n\in \mathbb{N}}\mathcal{P}_n$ be a
$\sigma$-HCP closed $k$-network of $\mathcal{F}_2(X)$, where each
$\mathcal{P}_n$ is HCP. Suppose $X$ is not an $\aleph$-space, then it follows from
\cite{JY1992} that $X$ contains a closed copy of
$$S_{\omega_1}=\{x\}\cup\{x_n(\alpha): n\in \mathbb{N}, \alpha\in
\omega_1\},$$ where $x_n(\alpha)\to x$ as $n\to \infty$ for each
$\alpha\in \omega_1$. In order to obtain a contradiction, we first prove the following two claims. For any $\beta<\omega_1$, we take any fixed $i_{\alpha}, j_{\alpha}\in\mathbb{N}$ for any $0<\alpha<\beta$.

\smallskip
{\bf Claim 1} For any $\beta<\omega_1$, $H_{\beta}=\{\{x_{i_\alpha}(\alpha),
x_{j_\alpha}(0)\}: 0<\alpha<\beta\}$ is closed discrete in
$\mathcal{F}_2(X)$.

\smallskip
Indeed, fix an arbitrary $\beta<\omega_1$, and take any $z\in
\mathcal{F}_2(X)$. If $z\in X\setminus \{x\}$, pick an open
neighborhood $U$ of $z$ in $X$ such that $|U\cap S_{\omega_1}|\leq
1$, then $\langle U\rangle\cap H_{\beta}=\emptyset$; if $z=x$, choose an
open neighborhood $U$ of $x$ such that each $x_{i_\alpha}(\alpha)\notin
U$, then $\langle U\rangle\cap H_{\beta}=\emptyset$; if $z=\{z_1, z_2\}\in
\mathcal{F}_2(X)\setminus X$, pick an open neighborhood $V_i $ of
$z_i$ in $X$ for each $i\leq 2$ such that either $|V_i\cap
S_{\omega_1}|\leq 1$ or $\{x_{i_\alpha}(\alpha): \alpha<\beta\}\cap
V_i=\emptyset$ for each $i\leq 2$. Then it is easy to check $|\langle
V_1, V_2\rangle \cap H_{\beta}|\leq 1$. Hence, $H_{\beta}$ is closed discrete.

\smallskip
{\bf Claim 2} For any $0<\alpha<\omega_1$,
$$K_\alpha=\{\{x_i(\alpha), x_j(0)\}, \{x_i(\alpha), x\}, \{x,
x_j(0)\}: i, j\in \mathbb{N}\}\cup \{\{x\}\}$$ is compact in
$\mathcal{F}_2(X)$.

\smallskip
Let $\mathcal{U}$ be an arbitrary open cover of $K_{\alpha}$, and let
$U\in \mathcal{U}$ with $\{x\}\in U$. Pick an open neighborhood $V$
of $x$ in $X$ such that $\langle V\rangle\subset U$. Obviously,
there exists $k\in\mathbb{N}$ such that $x_i(\alpha), x_j(0)\in V$
whenever $i, j\geq k$. Then $$L_1=\{\{x_i(\alpha), x_j(0)\}: i, j\geq
k\}\cup \{\{x\}\}\subset \langle V\rangle.$$ Note that
$\{x_i(\alpha), x_j(0)\}\to \{x_i(\alpha), x\}$ as $j\to \infty$,
$\{x_i(\alpha), x_j(0)\}\to \{x, x_j(0)\}$ as $i\to \infty$, it is
easy to see that $$L_2=\{\{x_i(\alpha), x_j(0)\}: i, j < k\},
L_3=\{\{x_i(\alpha), x_j(0)\}, \{x_i(\alpha), x\}: i< k,
j\in\mathbb{N}\}$$ and $$L_4=\{\{x_i(\alpha), x_j(0)\}, \{x, x_j(0)\}:
j< k, i\in\mathbb{N}\}$$ are compact in $\mathcal{F}_2(X)$. Hence
$K_\alpha=L_1\cup L_2\cup L_3\cup L_4$ could be covered by a finite
subfamily of $\mathcal{U}$. Therefore, $K_{\alpha}$ is compact in
$\mathcal{F}_2(X)$.

By induction, we can find a uncountable subset $\{P_\alpha: \alpha<\omega_1\}$ of $\mathcal{P}$ such that for any $\alpha\in\omega_{1}$ we have $\{x_{i_\alpha}(\alpha),
x_{j_\alpha}(0)\}\in P_\alpha$ for some $i_\alpha, j_\alpha\in
\mathbb{N}$ and $$|P_\alpha\cap \{\{x, x_j(0)\}: j\in
\mathbb{N}\}|=\omega.$$

\smallskip
Indeed, since $\mathcal{P}$ is a closed $k$-network, there is a finite subfamily
$\mathcal{P}_{1}$ of $\mathcal{P}$ which covers $K_1$, then it is easy to verify that there exists a
$P_1\in \mathcal{P}_{1}$ such that $\{x_{i_1}(1), x_{j_1}(0)\}\in P_1$
for some $i_1, j_1\in \mathbb{N}$ and $|P_1\cap \{\{x, x_j(0)\}:
j\in \mathbb{N}\}|=\omega$. Let $\alpha<\omega_1$ and suppose for each $\beta<\alpha$, we have
chosen $P_\beta\in \mathcal{P}$ such that $\{x_{i_\beta}(\beta),
x_{j_\beta}(0)\}\in P_\beta$ for some $i_\beta, j_\beta\in
\mathbb{N}$ and $|P_\beta\cap \{\{x, x_j(0)\}: j\in
\mathbb{N}\}|=\omega$. Since $\{\{x_{i_\beta}(\beta),
x_{j_\beta}(0)\}: \beta<\alpha\}$ is closed by Claim 1 and $$K_\alpha\subset
\mathcal{F}_2(X)\setminus \{\{x_{i_\beta}(\beta), x_{j_\beta}(0)\}:
\beta<\alpha\},$$ there is a finite subfamily $\mathcal{P}_{\alpha}\subset
\mathcal{P}$ with $$K_\alpha\subset \bigcup \mathcal{P}_{\alpha}\subset
\mathcal{F}_2(X)\setminus \{\{x_{i_\beta}(\beta), x_{j_\beta}(0)\}:
\beta<\alpha\},$$ then there is a $P_\alpha\in\mathcal{P}_{\alpha}$ such
that $\{x_{i_\alpha}(\alpha), x_{j_\alpha}(0)\}\in P_\alpha$ for
some $i_\alpha, j_\alpha\in \mathbb{N}$ and $|P_\alpha\cap \{\{x,
x_j(0)\}: j\in \mathbb{N}\}|=\omega$. Moreover, $P_{\alpha}\neq P_{\beta}$ for any $\beta<\alpha$.

\smallskip
Clearly, $\{P_\alpha: \alpha<\omega_1\}\subset \mathcal{P}$, hence there exists $n\in\mathbb{N}$
such that a countable infinite subfamily $\mathcal{B}$ of $\{P_\alpha:
\alpha<\omega_1\}$ is contained in $\mathcal{P}_n$. Let $\mathcal{B}=\{P_{n}: n\in\mathbb{N}\}$. For each $n\in\mathbb{N}$, since $|P_n\cap \{\{x, x_j(0)\}: j\in
\mathbb{N}\}|=\omega$, pick $$y_{n}=\{x, x_{j_n}(0)\}\in P_\alpha\cap \{\{x, x_j(0)\}: j\in
\mathbb{N}\}.$$ We may assume that $j_{n}<j_{n+1}$ for any $n\in\mathbb{N}$. Then $y_{n}\to x$ as $n\rightarrow\infty$. However, $\mathcal{P}_n$ is HCP,
then $\{y_{n}\}$ is closed discrete, this is a contradiction.

Therefore, $X$ contains no closed copy of $S_{\omega_1}$, that is, $X$ is an
$\aleph$-space.
\end{proof}

Next we discuss whether $\mathcal{K}(X)$ has a point-countable base whenever $X$ has a point-countable weak base. In order to solve this question, we first prove two propositions, which give the relations of a space $X$ and its hyperspace $\mathcal{K}(X)$ containing a copy of $S_{\omega}$ or $S_{2}$ ( see Propositions~\ref{p1} and~\ref{p2} respectively).

\begin{proposition}\label{p1}
If $X$ contains a closed copy of $S_\omega$, then $\mathcal{F}_2(X)$
contains a closed copy of $S_2$.
\end{proposition}

\begin{proof}
We write $S_\omega=\{x\}\cup\{x_n(m): m\in \omega, n\in
\mathbb{N}\}$, where each $x_n(m)$ is an isolated point of $X$ and the basic
neighborhoods of $x$ of the following form:
$$\{\{x\}\cup\{x_n(m): n\geq f(m)\}\}, \mbox{where}\ f\in\mathbb{N}^\omega.$$ Obviously, for any $m\in \omega$, $x_n(m)\to
x$ as $n\to \infty$. Next we construct a closed copy of $S_2$ in
$\mathcal{F}_2(X)$.

Let $y_n=\{x_n(0), x\}$ for each $n\in \mathbb{N}$,
$y_n(m)=\{x_m(0), x_n(m)\}$ for any $m, n\in \mathbb{N}$. It is easy
to see that $y_n=\{x_n(0), x\}\to x, y_n(m)\to y_m$ as $n\to
\infty$. Put
$$A=\{x\}\cup\{y_n: n\in \mathbb{N}\}\cup\{y_n(m): m, n\in
\mathbb{N}\}.$$
Now we claim that $A$ is a closed copy of $S_2$ in $\mathcal{F}_2(X)$. Indeed, since $\mathcal{F}_2(S_\omega)$ is closed in $\mathcal{F}_2(X)$, it suffices to prove that $A$ is a closed copy of $S_2$ in $\mathcal{F}_2(S_\omega)$.

First, we prove that $A$ is closed in $\mathcal{F}_2(S_\omega)$. Take any $z\in \mathcal{F}_2(S_\omega)\setminus A$, it suffices to find a
neighborhood of $z$ in $\mathcal{F}_2(S_\omega)$ which does not intersect
$A$. We divide the proof into the following two cases.

\smallskip
{\bf Case 1} $z\in S_\omega.$

\smallskip
Obviously, $z\neq x$, hence $\{z\}$ is an open neighborhood of $z$
in $S_\omega$, then $\langle
\{z\}\rangle$ is a neighborhood of $z$ in $\mathcal{F}_2(S_\omega)$ with
$\langle \{z\}\rangle \cap A=\emptyset$.

\smallskip
{\bf Case 2} $z\in \mathcal{F}_2(S_\omega)\setminus S_\omega.$

\smallskip
Then $z=\{z_1, z_2\}$, where $z_1, z_{2}\in S_\omega$.
If $x\not\in\{z_{1}, z_{2}\}$, then $\{z_1\}$ and $\{z_{2}\}$ are open neighborhoods of $z_1$ and $z_{2}$ in $S_\omega$ respectively. Clearly, $\langle \{z_1\},
\{z_{2}\}\rangle\cap A=\emptyset$ since $z\not\in A$ by our assumption; otherwise,
without loss of generality, we may assume that $z_1=x,
z_2=x_{n_0}(m_0)$ for some $m_0, n_0\geq 1$ since $z\in
\mathcal{F}_2(S_\omega)\setminus A$. Obviously, $S_\omega\setminus\{x_{n_0}(m_0), x_{m_{0}}(0)\}$ and $\{x_{n_0}(m_0)\}$ are open neighborhoods of $z_1$ and $z_2$ respectively in $S_\omega$, hence $$\langle S_\omega\setminus\{x_{n_0}(m_0), x_{m_{0}}(0)\},
\{x_{n_0}(m_0)\}\rangle\cap A=\emptyset.$$

\smallskip
Therefore, it follows from Cases 1 and 2 that $A$ is closed in $\mathcal{F}_2(S_\omega)$.

\smallskip
Next we prove that for any $f\in \mathbb{N}^\mathbb{N}$, the set
$B=\{y_n(m): n\leq f(m), m\in \mathbb{N}\}$ is discrete in
$\mathcal{F}_2(S_\omega)$. Since $A$ is closed, it suffices to find a
neighborhood $\widehat{U}$ of each point of $A$ such that
$|\widehat{U}\cap B|\leq 1$. Take an arbitrary point $z\in A$.

If $z=x$, then we can choose $g\in \mathbb{N}^\mathbb{N}$ such that
$g(i)>f(i)$ for each $i$. Let $$U=\{x\}\cup\{x_n(m): n\geq g(m),
m\in\omega\}.$$ Put $\widehat{U}=\langle U\rangle$. Then $\widehat{U}\cap B=\emptyset$. Now we only to consider the point $z\in
A\setminus\{x\}$. We write $z=\{z_1, z_2\}$. If there exists
$$z_i\notin \{x_n(m): n\leq f(m), m\in\mathbb{N}\}\cup\{x_n(0): n\in
\mathbb{N}\}$$ for some $i\in\{1, 2\}$, then without loss of
generality we assume that $$z_1\notin \{x_n(m): n\leq f(m),
m\in\mathbb{N}\}\cup\{x_n(0): n\in \mathbb{N}\}.$$ Therefore, there
are open neighborhoods $U, V$ of $z_1$ and $z_2$ in $S_\omega$ respectively
such that $U\cap \{x_n(i): n\leq f(i), i\in\mathbb{N}\}=\emptyset$
and $U\cap V=\emptyset$. Put $\widehat{U}=\langle U, V\rangle$. Thus
$\widehat{U}\cap B=\emptyset$. If $$\{z_1, z_2\}\subset \{x_n(m):
n\leq f(m), m\in\mathbb{N}\}\cup\{x_n(0): n\in \mathbb{N}\},$$ then $|\langle \{z_1\}, \{z_2\}\rangle\cap B|=1$.

\smallskip
In a word, $B$ is discrete and $A$ is a closed copy of $S_2$.
\end{proof}

By Proposition~\ref{p1}, we also can give a negative answer to \cite[Question 3.6]{GM2016}\footnote{In \cite{TLL2018}, the authors also gave a negative answer to this question.}, which shows that there exists a La\v{s}nev space $X$ such that $\mathcal{K}(X)$ is not a La\v{s}nev space. Moreover, it is well known that a space $X$ is a La\v{s}nev space if and only if $X$ is a Fr\'{e}chet-Urysohn space with a $\sigma$-compact-finite $k$-network.
Further, a space is $k^{\ast}$-metrizable if and only if it has a $\sigma$-compact-finite $k$-network \cite{BBK2008}. However, it follow from Lemma~\ref{l3} and Proposition~\ref{l4} that we have the following theorem, which was also proved in \cite{LT1998} and \cite{BBK2008} independently.

\begin{theorem}
A space $X$ is a $k^{\ast}$-metrizable space if and only if $\mathcal{K}(X)$ is a $k^{\ast}$-metrizable space.
\end{theorem}

\begin{proposition}\label{p2}
If $X$ contains a closed copy of $S_2$, then $\mathcal{K}(X)$
contains a closed copy of $S_\omega$.
\end{proposition}

\begin{proof}
Let $$S=\{y\}\cup\{y_n: n\in \mathbb{N}\}\cup\{y_i(n): i, n\in\mathbb{N}\}$$ be a
closed copy of $S_2$ in $X$, where $y_i(n)\to y_n$ as $i\to\infty$
for each $n\in \mathbb{N}$, $y_n\to y$ as $n\to\infty$ and for any
$f\in \mathbb{N}^\mathbb{N}$, $\{y_i(n): i\leq f(n), n\in
\mathbb{N}\}$ is discrete.

Let $K=\{y\}\cup \{y_n: n\in \mathbb{N}\}$, and for each $i,
n\in\mathbb{N}$, let $K_i(n)=\{y_i(n)\}\cup K$. Obviously, all $K,
K_i(n)\in \mathcal{K}(X)$, and $K_i(n)\to K$ as $i\to\infty$ for
each $n\in\mathbb{N}$. Let $$A=\{K\}\cup\{K_i(n): i,
n\in\mathbb{N}\}.$$ We claim that $A$ is a closed copy of $S_\omega$
in $\mathcal{K}(X)$.

Indeed, since $S$ is closed in $X$, it follows that $\mathcal{K}(S)$
is closed in $\mathcal{K}(X)$. In order to prove the closedness of
$A$, it suffices to show that $A$ is closed in $\mathcal{K}(S)$.

For $H\in \mathcal{K}(S)\setminus A$, we need to find a neighborhood
$\widehat{U}$ of $H$ in $\mathcal{K}(S)$ such that $\widehat{U}\cap
A=\emptyset$. We consider the following two cases.

\smallskip
{\bf Case 1} $K\setminus H\neq\emptyset$.

\smallskip
Pick $z\in K\setminus H$, and for each $x\in H$, take an open
neighborhood $V_x$ of $x$ in $X$ such that $z\notin V_x$. Then there
are finitely many $V_{x_i}(i\leq k)$ such that $H\subset
 \bigcup_{i\leq k}V_{x_i}$. Let $\widehat{U}=\langle V_{x_1}, ...,
 V_{x_k}\rangle$. Then it is easy to check that $\widehat{U}\cap
 A=\emptyset$.

\smallskip
{\bf Case 2} $K\subset H$.

\smallskip
Since $H\notin A$, it follows that $|H\cap\{y_i(n): i, n\in
\mathbb{N}\}|\geq 2$. Pick $z_1, z_2\in H\cap\{y_i(n): i, n\in
\mathbb{N}\}|$, and let $\widehat{U}=\langle S, \{z_1\},
\{z_2\}\rangle$. Then $\widehat{U}$ is an open neighborhood of $H$ and
$\widehat{U}\cap A=\emptyset$.

Therefore, $A$ is closed in $\mathcal{K}(S)$.

\smallskip
Next we prove that $A$ is a copy of $S_\omega$. Indeed, for any $f\in
\mathbb{N}^\mathbb{N}$, we claim that $B=\{K_i(n): n\in \mathbb{N}, i\leq
f(n)\}$ is discrete in $\mathcal{K}(S)$. Since $A$ is closed in
$\mathcal{K}(S)$, we only show that $B$ is discrete in $A$. Take an
arbitrary $H\in A$. Then it suffices to find a neighborhood
$\widehat{U}$ of $H$ in $\mathcal{K}(S)$ such that $\widehat{U}$ intersects at most one element of
$B$.
If $H\not\in B$, then let $U=S\setminus \{y_i(n): n\in \mathbb{N},
i\leq f(n)\}$. Obviously, $U$ is open in $S$, thus put
$\widehat{U}=\langle U\rangle$. It is easy to see that $\widehat{U}\cap
B=\emptyset$.  Now assume $H\in B$, then $H=K\cup\{y_i(n)\}$ for
some $i, n\in\mathbb{N}$. Let $\widehat{U}=\langle S,
\{y_i(n)\}\rangle$. Hence it is easy to see that $|\widehat{U}\cap
B|=1$.

\smallskip
In a word, $A$ is a closed copy of $S_\omega$ in $\mathcal{K}(X)$.
\end{proof}

The following theorem shows that if a space $X$ has a point-countable base then $\mathcal{K}(X)$ also has.

\begin{theorem}\label{tt1}
The following statements are equivalent for a space $X$.
\begin{enumerate}
\item $X$ has a point-countable base.

\item $\mathcal{F}_2(X)$ is a Fr\'echet-Urysohn space with a point-countable $k$-network.

\item $\mathcal{K}(X)$ is a Fr\'echet-Urysohn space with a point-countable $k$-network.

\item $\mathcal{K}(X)$ has a point-countable weak base.

\item $\mathcal{K}(X)$ has a compact-countable weak base.

\item $\mathcal{K}(X)$ has a point-countable base.
\end{enumerate}
\end{theorem}

\begin{proof}
(6) $\to$ (4), (6) $\to$ (3) and (3) $\to$ (2) are trivial. It is well known that a space with a point-countable base if and only if it has a compact-countable base, hence (6) $\to$ (5). Clearly, (5) $\to$ (6) is obvious if (4) $\to$ (6) holds. Therefore, it suffice to prove that (2) $\to$ (1), (1)
$\to$ (6) and (4) $\to$ (1).

(2) $\to$ (1). Since $X$ is a subspace of $\mathcal{F}_2(X)$, it
follows that $X$ has a point-countable $k$-network. We claim that
$X$ contains no closed copy of $S_\omega$; otherwise,
$\mathcal{F}_2(X)$ contains a closed copy of $S_2$ by Proposition
~\ref{p1}, which is a contradiction since any Fr\'echet-Urysohn space
contains no closed copy of $S_2$. Since a Fr\'echet-Urysohn space, which contains no closed copy of
$S_\omega$, is strongly Fr\'echet-Urysohn\footnote{A space $X$ is said to be {\it strongly
Fr$\acute{e}$chet-Urysohn} if the following condition is satisfied

(SFU) For every $x\in X$ and each sequence $\eta =\{A_{n}: n\in
\mathbb{N}\}$ of subsets of $X$ such that $x\in\bigcap_{n\in
\mathbb{N}}\overline{A_{n}}$, there is a sequence $\zeta =\{a_{n}:
n\in \mathbb{N\}}$ in $X$ converging to $x$ and intersecting
infinitely many members of $\eta$.}, the space $X$ is strongly Fr\'echet-Urysohn, hence $X$
has a point-countable base by \cite{GMT1984}.

(1) $\to$ (6). Since $X$ has a point-countable base, each compact
subset of $X$ is metrizable \cite{G1984}, hence each compact subset is separable;
moreover, it is easy to see that $X$ has a compact-countable $k$-network, thus $\mathcal{K}(X)$ has a compact-countable $k$-network by Proposition~\ref{l4}; further, it is obvious that each compact subset of $X$ has a countable character in $X$, which implies that
$\mathcal{K}(X)$ is first-countable. Since a first-countable space
with a point-countable $k$-network has a point-countable base
\cite{GMT1984}, $\mathcal{K}(X)$ has a point-countable base.

(4) $\to$ (1). Since $X$ is a closed subspace of
$\mathcal{K}(X)$, it follows that $X$ has a point-countable weak base. We claim
that $X$ contains no closed copy of $S_2$. Suppose not, we assume
that $X$ has a closed copy of $S_2$, then from Proposition ~\ref{p2} it follows that
$\mathcal{K}(X)$ contains a closed copy of $S_\omega$, this is a
contradiction since a $g$-first-countable space does not contain any
closed copy of $S_\omega$.
Since $X$ is a space with a point-countable weak base and contains
no closed copy of $S_2$, it is Fr\'echet-Urysohn. However, a
$g$-first-countable, Fr\'echet-Urysohn space is first-countable
and a first-countable space with a point-countable weak base has a
point-countable base, hence $X$ has a point-countable base.
\end{proof}

Note that a first-countable space with a $\sigma$-locally countable
$k$-network has a $\sigma$-locally countable base \cite{B1980}, we
can use the same proof of Theorem~\ref{tt1} to prove the following two theorems.

\begin{theorem}
The following statements are equivalent for a space $X$.
\begin{enumerate}
\item $X$ has a $\sigma$-locally countable base;

\item $\mathcal{K}(X)$ has a $\sigma$-locally countable weak base;

\item $\mathcal{K}(X)$ has a $\sigma$-locally countable base.

\end{enumerate}
\end{theorem}

\begin{theorem}
The following statements are equivalent for a space $X$.
\begin{enumerate}
\item $X$ is metrizable;

\item $\mathcal{K}(X)$ has a $\sigma$-locally finite weak base;

\item $\mathcal{K}(X)$ is metrizable.

\end{enumerate}
\end{theorem}

Finally we prove that if $X$ is an $\alpha$-space\footnote{A space $X$ is an $\alpha$-space \cite{GM2016}
($\sigma^{\sharp}$-space in sense of \cite{H1971}) if there exists a
function $g: \mathbb{N}\times X\to \tau_X$ such that (i)
$\bigcap_{m\in\mathbb{N}}g(m, x)=\{x\}$, (ii) if $y\in g(m, x)$,
then $g(m, y)\subset g(m, x)$.} then $\mathcal{K}(X)$ is also an
$\alpha$-space.

\begin{theorem}
$X$ is an $\alpha$-space if and only if $\mathcal{K}(X)$ is an
$\alpha$-space.
\end{theorem}

\begin{proof}
Assume that $X$ is an $\alpha$-space, then there exists a
$g$-function $g: \omega\times X\to \tau_X$ satisfying the
definition of $\alpha$-space; moreover, it follows from \cite[Lemma 3.22]{GM2016} that
we may assume that $g(n+1, x)\subset g(n, x)$ for each $x\in X$ and
$n\in\mathbb{N}$. We first prove the following claim.

\smallskip
{\bf Claim:} For any $K\in \mathcal{K}(X)$, we have $\bigcap_{n\in
\omega}(\bigcup_{x\in K}g(n, x))=K$.

\smallskip
It suffices to prove that, for any $y\notin K$, there exists
$k\in\mathbb{N}$ such that $y\notin \bigcup_{x\in K}g(k, x)$.
Obviously, for each $x\in K$, we can choose $n(x)\in \mathbb{N}$
such that $y\notin g(n(x), x)$, then $y\notin \bigcup_{x\in
K}g(n(x), x)$. Hence, from the compactness of $K$, there exists
$m\in\mathbb{N}$ such that $$K\subset \bigcup_{i\leq m}g(n(x_i),
x_i)\subset \bigcup_{x\in K}g(n(x), x).$$ Let $k=\max\{n(x_i): i\leq
m\}$. Then, for each $x\in K$, we have $x\in g(n(x_i), x_i)$ for some $i\leq
m$, $g(n(x_i), x)\subset g(n(x_i), x_i)$ and $g(k, x)\subset g(n(x_i),
x)$, hence $$K\subset \bigcup_{x\in K}g(k, x)\subset \bigcup_{i\leq
m}g(n(x_i), x_i)\ \mbox{and}\ y\notin \bigcup_{x\in K}g(k, x).$$ Therefore,
$\bigcap_{n\in \omega}(\bigcup_{x\in K}g(n, x))=K$.

For each $n\in\mathbb{N}$ and $K\in\mathcal{K}(X)$, choose $\{x_i(n,
K): i\leq n(K)\}\subset K$ for some $n(K)\in\mathbb{N}$ such that $\{g(n, x_i(n, K)): i\leq
n(K)\}$ is a minimal cover of $K$. Define $G: \omega\times
\mathcal{K}(X)\to \tau_{\mathcal{K}(X)}$ as follows:
$$G(n, K)=\langle g(n, x_1(n, K)), ..., g(n, x_{n(K)}(n, K))\rangle.$$
We claim that $G$ is a function satisfying the
definition of $\alpha$-space.

\smallskip
(a) $\bigcap_{n\in \omega}\{G(n, K): n\in\omega\}=K$.

\smallskip
In fact, take an arbitrary $H\in \mathcal{K}(X)\setminus\{K\}.$ If
$H\setminus K\neq \emptyset$, pick $y\in H\setminus K$. By Claim,
there exists $n_0\in\mathbb{N}$ such that $y\notin \bigcup_{x\in
K}g(n_0, x)$, then $\{y\}\notin G(n_0, K)$, thus $H\notin
\bigcap_{n\in \omega}G(n, K)$. If $H\setminus K=\emptyset$, pick
$y\in K\setminus H$. Hence, for each $x\in K$, there exists $n(x)\in
\omega$ such that $g(n(x), x)$ can not meet both $H$ and $\{y\}$.
From the compactness of $K$, there exists $m\in\mathbb{N}$ such that
$K\subset \bigcup_{i\leq m}g(n(x_i), x_i)$. Let $k=\max\{n(x_i):
i\leq m\}$. Then $K\subset \bigcup_{x\in K}g(k, x)$. Since
$\{x_{i}(k, K): i\leq k(K)\}\subset K\subset \bigcup_{x\in K}g(k,
x)$ and the assumption of each $g(k, x)$, each $g(k, x_i(k, K))
(i\leq k(K))$ can not meet both $H$ and $\{y\}$, that is, $H\notin
G(k, K)$. Therefore, $H\notin \bigcap_{n\in \omega}\{G(n, K):
n\in\omega\}$.

\smallskip
(b) For each $n\in\omega$ and $H\in G(n, K))$, then $G(n, H)\subset
G(n, K)$.

\smallskip
 Take an any $H\in G(n, K)=\langle g(n, x_1(n, K)), ..., g(n, x_{n(K)}(n, K))\rangle$. Obviously, $$H\in G(n, H)=\langle g(n, x_1(n, H)), ..., g(n,
 x_{n(H)}(n, H))\rangle.$$
 For any $y\in H$, we have $y\in g(n, x_i(n, K))$ for some $i\leq n(K)$, then $g(n, y)\subset g(n,
 x_i(n, K)$. Since $\{g(n, x_1(n)), ..., g(n,
 x_{m_n}(n))\}$ is a minimal cover of $K$, for each $g(n, x_i(n, K)$
 there is a $g(n, x_j(n, H))$ such that $g(n, y_j(n, H))\subset g(n,
 x_i(n, K))$. By Lemma ~\ref{l1}, $$\langle g(n, x_1(n, H)), ..., g(n,
 y_{n(H)}(n, H))\rangle \subset \langle g(n, x_1(n, K)), ..., g(n,
 x_{n(K)}(n, K))\rangle,$$ thus $G(n, H)\subset G(n, K)$.
\end{proof}

\maketitle
\section{Generalized metric properties on $\mathcal{F}(X)$}
In this section, we mainly discuss the generalized metric properties on $\mathcal{F}(X)$, and study the relations of that the topological property $\mathcal{P}$ in $X$ and $\mathcal{F}(X)$. First, we recall some concepts.

A space $X$ is said to have a {\it base of countable order} (resp., {\it $W_{\delta}$-diagonal}) if there is a sequence
$\{\mathcal{B}_n\}$ of bases for $X$ such that: Whenever $x\in
b_n\in\mathcal{B}_n$ and $\{b_n\}$ is decreasing (by set inclusion),
then $\{b_n: n\in \omega\}$ is a base at $x$ (resp., $\{x\}=\bigcap\{b_n: n\in \omega\}$). We use `BCO' to abbreviate `base of countable order'.

The following theorem shows that if $X$ has a BCO then $\mathcal{F}(X)$ also has.

\begin{theorem}\label{bco}
A space $X$ has a BCO if and only if $\mathcal{F}(X)$ has a BCO.
\end{theorem}

\begin{proof}
Since $X$ is a subspace of $\mathcal{F}(X)$, it suffices to prove
the necessity. Assume that $X$ have a BCO, then there is a squence
$\{\mathcal{B}_n\}$ of bases in $X$ satisfying the conditions of definition of BCO above.

For each $n\in\mathbb{N}$, let $$\mathcal{P}_n=\{\langle B_1, ...,
B_{m}\rangle: m\in \mathbb{N}, B_i\cap B_j=\emptyset (i\neq
j),B_i, B_{j}\in \mathcal{B}_n, i, j\leq m\}.$$ Then $\mathcal{P}_n$ is a
base of $\mathcal{F}(X)$.
Indeed, for any $\{x_1,..., x_r\}\in \widehat{U}$ with $\widehat{U}$ open in
$\mathcal{F}(X)$, let $\langle V_1,..., V_r\rangle\subset \widehat{U}$
with $x_i\in V_i, i\leq r$, $V_i\cap V_j=\emptyset$ if $i\neq j$.
Since $\mathcal{B}_n$ is a base of $X$, for each $i\leq r$, there
exists $P_i\in\mathcal{B}_n$ such that $x_i\in P_i\subset V_i$, then
$\langle P_1,..., P_r\rangle\in \mathcal{P}_n$ and $$\{x_1,...,
x_r\}\in \langle P_1,..., P_r\rangle\subset \widehat{U}.$$

Let $\{x_1,..., x_r\}\in \widehat{b_n}\in \mathcal{P}_n$, where
$\widehat{b_n}=\langle P_1(n),..., P_{r_n}(n)\rangle$ and
$\{\widehat{b_n}\}$ is decreasing. Then it is easy to see that $r_n\leq
r$. Now we prove $\{\widehat{b_n}\}$ is a base at $\{x_1,..., x_r\}$ in
$\mathcal{F}(X)$.

Since $\{\widehat{b_n}\}$ is decreasing, it follows from
Lemma~\ref{l1} and the construction of each $\mathcal{P}_n$ that the
sequence $\{r_{n}\}_{n\in\mathbb{N}}$ is increasing. Hence there
exists $k\in \mathbb{N}$ such that $r_n=r_k$ whenever $n\geq k$. We
claim that $r_k=r$.

Suppose not, then $r_n<r$ for all $n\in \mathbb{N}$. From the
construction of each $\mathcal{P}_n$, there are a decreasing
sequence $\{P_{i_n}(n): n\in \omega\}$ and two distinct points $x_p,
x_q\in \{x_1,..., x_r\}$ such that $\{x_p, x_q\}\subset P_{i_n}(n)$
for each $n\in\mathbb{N}$, where each $i_{n}\leq r_{n}$. Since $X$
has a BCO, then $\{P_{i_n}(n): n\in \omega\}$ is a base at $x_p$ and
$x_q$. This is impossible since $X$ is Hausdorff.

Now all $r_n=r$ for any $n\geq k$, hence $\widehat{b_n}=\langle
P_1(n),..., P_{r_n}(n)\rangle = \langle P_1(n),..., P_r(n)\rangle$.
For any $\{x_1,..., x_r\}\in \widehat{U}$ with $\widehat{U}$ open in
$\mathcal{F}(X)$, we find open neighborhoods $V_i$ of $x_i$ in $X$ for each
$i\leq r$ such that $V_i\cap V_j=\emptyset$ if $i\neq j$ and
$\langle V_1,..., V_r\rangle\subset \widehat{U}$. Fix an arbitrary
$i\leq r$. For any $n\geq k$, pick $P_{j_i}(n)\in \{P_1(n),...,
P_r(n)\}$ such that $x_i\in P_{j_i}(n)$. Then $\{P_{j_i}(n): n\geq
k\}$ is a base at $x_i$, hence there exists $m_i\in\mathbb{N}$ such
that $P_{j_i}(n)\subset V_i$ whenever $n\geq m_i$. Let $m=\max\{m_i:
i\leq r\}$. Then $\{x_1,..., x_r\}\in b_n=\langle P_1(n),...,
P_r(n)\rangle\subset \langle V_1, ..., V_r\rangle\subset\widehat{U}$
whenever $n\geq m$. Therefore, $\{\widehat{b_n}\}$ is a base at
$\{x_1,..., x_r\}$ in $\mathcal{F}(X)$.
\end{proof}

By a similar proof, we have the following theorem.

\begin{theorem}\label{Wd}
A space $X$ has a $W_{\delta}$-diagonal if and only if $\mathcal{F}(X)$ has a $W_{\delta}$-diagonal.
\end{theorem}

From \cite{AJ2000}, we know that each space with a sharp base\footnote{A {\it sharp base} $\mathscr{B}$ of a space $X$ \cite{AJ2000} is a base of $X$ such
that, for every injective sequence $\{B_{n}\}_{n\in\mathbb{N}}\subset\mathscr{B}$ and every $x\in\bigcap_{n\in\mathbb{N}}B_{n}$, the sequence $\{\bigcap_{i\leq n}B_{i}\}_{n\in\mathbb{N}}$ is a base at $x$.} has a BCO. However, we have the following example.

\begin{example}
There exists a space $X$ with a sharp base such that $\mathcal{F}_{2}(X)$ does not have any sharp base, hence $\mathcal{F}(X)$ does not have any sharp base.
\end{example}

\begin{proof}
Indeed, let $Y$ be the space in \cite[Lemma 1]{BG2005}, and let $X$ be the topological sum of $Y$ with the unit interval [0, 1] (that is, $X=Y\oplus [0, 1]$). Clearly, $X$ has a sharp base; however, it follows from
\cite[Lemma 1]{BG2005} that $X\times X$ does not have any sharp base since $Y\times [0, 1]$ does not have any sharp base. Assume that $\mathcal{F}_{2}(X)$ has a sharp base. Then $\mathcal{F}_{2}(X)\setminus \mathcal{F}_{1}(X)$ has a sharp base since a space with a sharp base is hereditary. Put $\bigtriangleup=\{(x, x): x\in X\}$. Then the mapping $$f: (X\times X)\setminus\bigtriangleup\longrightarrow \mathcal{F}_{2}(X)\setminus \mathcal{F}_{1}(X)$$ given by $f(x, y)=\{x, y\}$ is open two-to-one. Then $(X\times X)\setminus\bigtriangleup$ has a sharp base since a space with a sharp base is inverse preserved by open two-to-one mapping \cite{LH2005}. Therefore, $Y\times [0, 1]$ must have a sharp base since it is a subspace of $X\times X\setminus\bigtriangleup$, which is a contradiction.
\end{proof}

The following question is interesting and remains open.

\begin{question}
Let $X$ be a space such that $X^{n}$ has a sharp base for each $n\in\mathbb{N}$. Does $\mathcal{F}(X)$ have a sharp base?
\end{question}

Next we discuss the hyperspaces $\mathcal{F}(X)$ on a Nagata space $X$, a space $X$ with a point-regular base, $c$-semi-stratifiable space $X$ and $\gamma$-space, respectively.

A space $X$ is called {\it Nagata space} if for any $x\in X$, there
are two sequences $\{U_n(x): n\in \mathbb{N}\}$ and
$\{V_n(x): n\in \mathbb{N}\}$ of open neighborhoods of $x$ in $X$ such that
the following conditions hold:

\smallskip
(i)$\{U_n(x): n\in \mathbb{N}\}$ is a local base at $x$ such that
$U_{n+1}(x)\subset U_{n}(x)$ for each $n\in\mathbb{N}$;

\smallskip
(ii) if $y\notin U_n(x)$, then $V_n(x)\cap V_n(y)=\emptyset$.

\begin{theorem}
A space $X$ is Nagata if and only if $\mathcal{F}(X)$ is a Nagata
space.
\end{theorem}

\begin{proof}
Since $X$ is a subspace of $\mathcal{F}(X)$, it suffices to prove
the necessity. Suppose $X$ is a Nagata space. For each $x\in X$, let
$\{U_n(x): n\in \mathbb{N}\}$ and $\{V_n(x): n\in \mathbb{N}\}$ be
the sequences of open neighborhoods of $x$ in $X$ satisfying the
conditions of the definition of Nagata space.

Take any point $\{x_1,..., x_r\}\in \mathcal{F}(X)$. Define
$$U(n, \{x_1,.., x_r\})=\langle U_n(x_1),..., U_n(x_r)\rangle$$ and $$V(n, \{x_1,...,
x_r\})=\langle V_n(x_1),..., V_n(x_r)\rangle.$$ Then, by a similar
proof of Theorem ~\ref{bco}, we could prove $\{U(n, \{x_1,..,
x_r\}): n\in \mathbb{N}\}$ is a local base at $\{x_1,..., x_r\}$ in
$\mathcal{F}(X)$.

Next we prove $\{U(n, \{x_1,.., x_r\}): n\in \mathbb{N}\}$ and
$\{V(n, \{x_1,.., x_r\}): n\in \mathbb{N}\}$ satisfy (ii) of the
definition of Nagata space.

Assume that $\{y_1,..., y_s\}\notin U(n, \{x_1,.., x_r\})$, then we
prove that $$V(n, \{y_1,..., y_s\})\cap V(n, \{x_1,...,
x_r\})=\emptyset.$$ We consider the following two cases:

\smallskip
(a) there exists $y_j\in \{y_1,..., y_s\}$ such that $y_j\notin
U_n(x_i)$ for all $i\leq r$.

\smallskip
Assume that $V(n, \{y_1,..., y_s\})\cap V(n, \{x_1,...,
x_r\})\neq\emptyset$. Take any $$\{z_1, ..., z_t\}\in V(n,
\{y_1,..., y_s\})\cap V(n, \{x_1,..., x_r\}).$$ Then there is $k$
such that $z_k\in V_n(y_j)$; on the other hand, $z_k\in
\bigcup_{i\leq r}V_n(x_i)$, which is a contradiction since
$V_n(y_j)\cap V_n(x_i)=\emptyset$ for all $i\leq r$ by (ii) of the
definition of Nagata space.

\smallskip
(b) there exists $i\leq r$ such that $\{y_1,..., y_s\}\cap
U_n(x_i)=\emptyset$.

\smallskip
By (ii) of the definition of Nagata space, we have $\bigcup_{j\leq
s}V_n(y_j)\cap V_n(x_i)=\emptyset$. Suppose $z\in V(n, \{y_1,...,
y_s\})\cap V(n, \{x_1,..., x_r\})$, then $z\subset \bigcup_{j\leq
s}V_n(y_j)$ and $z\cap V_n(x_i)\neq \emptyset$. This is a
contradiction.

Therefore, $\mathcal{F}(X)$ is a Nagata space.
\end{proof}

A base $\mathcal{B}$ for a pace $X$ is {\it
point-regular} \cite {E1989} if for every point $x\in X$ and any
neighborhood $U$ of $x$ the set of all members of $\mathcal{B}$ that
contain $x$ and meet $X\setminus U$ is finite. A space having
point-regular base is characterized as follows.

\begin{theorem}
\cite[Lemma 5.4.7]{E1989} A space $X$ has a point-regular base if
and only if $X$ has a development consisting of point-finite covers.
\end{theorem}

\begin{theorem}\label{point-regular}
A space $X$ has a point-regular base if and only if $\mathcal{F}(X)$
has a point-regular base.
\end{theorem}

\begin{proof}
Since $X$ is a subspace of $\mathcal{F}(X)$, it suffices to prove
the necessity. Suppose $X$ has a point-regular base, by Thereom
~\ref{point-regular}, let $\{\mathcal{U}_n: n\in \mathbb{N}\}$ be a
development consisting of point-finite covers, and let
$$\mathcal{G}_n=\{\bigcap_{i\leq n}U_i: U_i\in\mathcal{U}_i, i\leq
n\}.$$

It is easy to see that $\{\mathcal{G}_n: n\in \mathbb{N}\}$ is a development consisting
of point-finite cover of $X$ with $st(x, \mathcal{G}_{m+1})\subset
st(x, \mathcal{G}_m)$ for all $x\in X$.

For each $n\in\mathbb{N}$, put $$\mathcal{W}_n=\{\langle B_1(n),...,
B_{m}(n)\rangle: B_i(n)\in \mathcal{G}_n, i\leq m,
m\in\mathbb{N}\}.$$ Obviously, each $\mathcal{W}_n$ is an open cover
of $\mathcal{F}(X)$. We claim that $\{\mathcal{W}_n: n\in
\mathbb{N}\}$ is a development consisting of point-finite covers of
$\mathcal{F}(X)$.

\smallskip
(1) For each $n\in \mathbb{N}$, $\mathcal{W}_n$ is point-finite.

\smallskip
In fact, take any point $\{x_1,..., x_r\}\in \mathcal{F}(X)$. If
$\{x_1,..., x_r\}\in \langle B_1(n),..., B_{m}(n)\rangle$, then for
each $i\leq m$, $B_i(n)$ contains some $x_j$. However, since
$\mathcal{G}_n$ is point-finite, we have $$|\{B\in\mathcal{G}_n:
B\cap \{x_1,..., x_r\}\neq \emptyset\}|<\omega.$$ Therefore, there
are only finitely many elements in $\mathcal{W}_n$ intersecting
$\{x_1,..., x_r\}$.

\smallskip
(2) $(\mathcal{W}_n)$ is a development.

\smallskip
Take any $\{x_1,..., x_r\}\in \widehat{U}$ with $\widehat{U}$ open in
$\mathcal{F}(X)$. Then we may find $\{V_i: i\leq r\}$ with each
$V_i$ open in $\mathcal{F}(X)$, $x_i\in V_i$ and $V_i\cap
V_j=\emptyset (i\neq j)$ such that $\langle V_1,...,
V_r\rangle\subset \widehat{U}$. Since $\{\mathcal{G}_n: n\in \mathbb{N}\}$ is a development
of $X$, for each $i$, there is $m_i\in\mathbb{N}$ such that $x_i\in st(x_i,
\mathcal{G}_n)\subset V_i$ whenever $n\geq m_i$. Let $m=\max\{m_i:
i\leq r\}$. Then $$\{x_1,..., x_r\}\in \langle st(x_1,
\mathcal{G}_n), ..., st(x_r, \mathcal{G}_n)\rangle\subset \langle
V_1,..., V_r\rangle$$ whenever $n\geq m$.

Therefore, $\mathcal{F}(X)$ has a development consisting of
point-finite cover, so that $\mathcal{F}(X)$ has a point-regular
base by Theorem~\ref{point-regular}.
\end{proof}

A $c$-semi-stratifiable\footnote{A space $X$ is said to be {\it c-semi-stratifiable}
\cite{M1973} if, for every compact subset $C$ of $X$, there is a
sequence $\{G(n, C)\}$ of open sets of $X$ such that (1)
$C=\bigcap_{n\in \omega}G(N, C)$; (2) $C\subset F \Rightarrow G(n,
C)\subset G(n, F)$.} is characterized as
follow.

\begin{lemma}\label{l5}
\cite{BBL2006} A space is $c$-semi-stratifiable if and only if there
is a $g$-function on $X$ such that (1) $\{x\}=\bigcap_{n\in
\omega}g(n, x), x\in X$; (2) if a sequence $\{x_n\}$ of distinct
points of $X$ converges to some $x\in X$, then $\bigcap_{n\in
\omega}g(n, x_n)\subset \{x\}$.
\end{lemma}

\begin{theorem}
A space $X$ is $c$-semi-stratifiable if and only if $\mathcal{F}(X)$
is a $c$-semi-stratifiable space.
\end{theorem}

\begin{proof}
Suppose that $(X, \tau)$ is a $c$-semi-stratifiable space, and let
$g: \omega\times X\to \tau_X$ be the $g$-function in Lemma
~\ref{l5}. Define $G: \omega\times \mathcal{F}(X)\to
\tau_{\mathcal{F}(X)}$ as follow: $$G(n, \{x_1,..., x_r\})=\langle
g(n, x_1),..., g(n, x_r)\rangle.$$  We claim that $G$ satisfies
the conditions of Lemma~\ref{l5} as follows.

\smallskip
(1) $\bigcap_{n\in \omega}G(n,\{x_1,..., x_r\})=\{x_1,..., x_r\}$.

\smallskip
Take any $\{y_1,..., y_s\}\neq \{x_1,..., x_r\}$. Then $$\{y_1,...,
y_s\}\setminus \{x_1,..., x_r\}\neq\emptyset\ \mbox{or}\ \{y_1,...,
y_s\}\setminus\{x_1,..., x_r\}=\emptyset.$$ If $\{y_1,...,
y_s\}\setminus \{x_1,..., x_r\}\neq\emptyset$, then take any
$y_{i}\in\{y_1,..., y_s\}\setminus \{x_1,..., x_r\}$. Then, for
sufficiently large $n\in\mathbb{N}$, we have
$y_{i}\not\in\bigcup_{j=1}^{r}g(n, x_j)$, which shows that
$$\{y_1,..., y_s\}\not\in G(n,\{x_1,..., x_r\}).$$ If $\{y_1,...,
y_s\}\setminus\{x_1,..., x_r\}=\emptyset$, then there exists $j\leq
r$ such that $x_{j}\not\in\{y_1,..., y_s\}$. Then, for sufficiently
large $m\in\mathbb{N}$, we have $\{y_1,..., y_s\}\cap g(n,
x_j)=\emptyset$, which shows that $\{y_1,..., y_s\}\not\in
G(n,\{x_1,..., x_r\})$. Therefore, $\bigcap_{n\in
\omega}G(n,\{x_1,..., x_r\})=\{x_1,..., x_r\}$.

\smallskip
(2) If a sequence $\{Y_n\}$ of distinct points of $\mathcal{F}(X)$
converges to some $\{x_1,..., x_r\}\in \mathcal{F}(X)$, then
$\bigcap_{n\in \omega}G(n, Y_n)\subset \{\{x_1,..., x_r\}\}$.

\smallskip
We write $Y_n=\{y_1(n), ..., y_{m_n}(n)\}$ for each
$n\in\mathbb{N}$, where $m_n\in \mathbb{N}$.

Let $\{V_i: i\leq r\}$ be the family of open subsets of $X$ with $x_i\in V_i$,
$V_i\cap V_j=\emptyset$ for any $i, j\leq r$. Obviously, $\langle
V_1\cap g(n, x_1),..., V_r\cap g(n, x_r)\rangle$ is an open
neighborhood of $\{x_1,..., x_r\}$, since $Y_n\to \{x_1,..., x_r\}$
as $n\rightarrow\infty$, there exists $q\in \mathbb{N}$ such that
$Y_n\in \langle V_1\cap g(n, x_1),..., V_r\cap g(n, x_r)\rangle$
whenever $n\geq q$. Note that $\{V_i: i\leq r\}$ are disjoint
family, $m_n\geq r$ as $n\geq q$.

Suppose $$\{y_1,..., y_s\}\in \bigcap_{n\in\mathbb{N}}G(n,
Y_n)=\bigcap_{n\in\mathbb{N}}\langle g(n, y_1(n)),..., g(n,
y_{m_{n}}(n))\rangle.$$ Fix an arbitrary $j\leq s$. we prove that
$y_{j}=x_{i}$ for some $i\leq r$. Indeed, for each $n\in\mathbb{N}$,
we have $y_j\in g(n, y_{i_j(n)}(n))$ for some $i_{j}(n)\leq m_{n}$.
Then there exist an increasing sequence $\{n_k\}\subset\mathbb{N}$
and $i\leq r$ such that $n_1\geq q, y_j\in g(n_k,
y_{i_j(n_{k})}(n_k))$ and $y_{i_{j}(n_{k})}(n_k)\in V_i\cap g(n_{k},
x_i)$. We claim that $y_{i_j(n_k)}(n_k)\to x_i$ as $n_k\to\infty$.
Indeed, let $U$ be an open neighborhood of $x_i$. Hence $\langle V_1\cap g(1,
x_1),..., U\cap V_i\cap g(1, x_i),..., V_r\cap g(1, x_r)\rangle$ is
an open neighborhood of $\{x_1,..., x_r\}$, then there exists
$q_1\in \mathbb{N}$ such that $$Y_n\in \langle V_1\cap g(1, x_1),...,
U\cap V_i\cap g(1, x_i),..., V_r\cap g(1, x_r)\rangle$$ as $n\geq
q_1$. It implies that $$y_{i_j(n_k)}(n_k)\in U\cap V_i\cap g(1,
x_i)\subset U$$ for $n_{k}\geq q_1$. Hence $y_{i_j(n_k)}(n_k)\to x_i$
as $n_k\to\infty$.

Therefore, we have $y_j\in\bigcap_{k\in \omega}g(n_k,
y_{i_j(n_k)}(n_k))\subset \{x_i\}$, thus $y_j=x_i$, so that $r\geq
s, \{y_1,..., y_s\}\subset\{x_1,..., x_r\}$.

Next, we prove $r=s$. Suppose not, $s<r$, hence $r\geq 2$. Since $m_n\geq r$  and
$g(n, y_i(n))\cap \{y_1,..., y_s\}\neq \emptyset$ for all $n\geq q,
i\leq m_n$, then there exist $y_m\in \{y_1, ..., y_s\}$, two distinct $p_1, p_2\leq
r$ and an increasing sequence $\{n_k\}\subset\mathbb{N}$ such that
$$n_1\geq q, y_m\in g(n_k, y_{i_{p_1}(n_k)}(n_k))\cap g(n_k,
y_{i_{p_2}(n_k)}(n_k))$$ and $y_{i_{p_j(n_k)}}(n_k)\in V_{p_j}\cap
g(n_{k}, x_{p_j}), x_{p_j}\in V_{p_j}$ for any $k\in\mathbb{N}$ and $j=1, 2$. By the claim above,
$y_{i_{p_j(n_k)}}(n_k)\to x_{p_j}$ as $k\rightarrow\infty$ for $j=1, 2$. Then $$y_m\in \bigcap_{k\in
\mathbb{N}} g(n_k, y_{i_{p_j}(n_k)}(n_k))\subset \{x_{p_j}\},
j=1,2,$$ hence $y_m=x_{p_1}, y_m=x_{p_2}$, this is a contradiction.

Therefore, $\{y_1,..., y_s\}=\{x_1,..., x_r\}$.
\end{proof}

\begin{theorem}
A space $X$ is a $\gamma$-space\footnote{A space $(X, \tau)$ is a $\gamma$-space if  there exists a function
$g: \omega\times X\to \tau_X$ such that (i) $\{g(n, x):
n\in\omega\}$ is a base at $x$; (ii) for each $n\in\omega$ and $x\in
X$, there exists $m\in\omega$ such that $y\in g(m, x)$ implies $g(n,
y)\subset g(n, x)$.} if and only if $\mathcal{F}(X)$ is a
$\gamma$-space.

\end{theorem}

\begin{proof}
Assume that $X$ is a $\gamma$-space. Hence there exists a
$g$-function $g: \omega\times X\to \tau_X$ satisfying the definition
above. We also may assume that $g(n+1, x)\subset g(n, x)$ for each
$x\in X$ and $n\in \omega$. Define $G: \omega\times
\mathcal{F}(X)\to \tau_{\mathcal{F}(X)}$ by $$G(n, \{x_1...,
x_r\})=\langle g(n, x_1),..., g(n, x_r)\rangle$$ for each
$\{x_1,..., x_r\}\in \mathcal{F}(X)$. We claim that the function $G$
satisfies the definition above.

\smallskip
(a) For each $\{x_1,..., x_r\}\in \mathcal{F}(X)$, the family
$\{G(n, \{x_1..., x_r\}): n\in\mathbb{N}\}$ is a local base at
$\{x_1,..., x_r\}$ in $\mathcal{F}(X)$.

\smallskip
Indeed, take an any open neighborhood $\widehat{U}$ of $\{x_1,...,
x_r\}$ in $\mathcal{F}(X)$. Then there is an open neighborhood
$\langle U_1,..., U_r\rangle$ of $\{x_1,..., x_r\}$ such that each
$U_i$ is an open neighborhood of $x_{i}$ in $X$, $U_i\cap
U_j=\emptyset$ ($i\neq j$) and $\langle U_1,..., U_r\rangle\subset
\widehat{U}$. Obviously, for each $i\leq r$, there exists $n_i$ such
that $g(k, x_i)\subset U_i$ whenever $k\geq n_i$. Let $m=\max\{n_i:
i\leq r\}$. Then $$\{x_1,..., x_r\}\in G(m, \{x_1,..., x_r\})=\langle
g(m, x_1),..., g(m, x_r)\rangle\subset \widehat{U}.$$

\smallskip
(b) For $\{x_1,..., x_r\}\in \mathcal{F}(X)$ and $n\in\mathbb{N}$,
there exists $m$ such that if $\{y_1,..., y_k\}\in g(m, \{x_1,...,
x_r\})$, then $g(m, \{y_1,..., y_k\})\subset g(n, \{x_1,...,
x_r\})$.

\smallskip
For $\{x_1,..., x_r\}\in \mathcal{F}(X)$ and $n\in\mathbb{N}$, since
$X$ is a $\gamma$-space, there exists $m\in\mathbb{N}$ such that,
for any $i\leq r$, if $y\in g(m, x_{i})$ then $g(m, y)\subset g(n,
x_{i})$. Take an arbitrary $\{y_1,..., y_k\}\in G(m, \{x_1,...,
x_r\})$. Then it follows from Lemma~\ref{l1} that
$$G(m, \{y_1, ..., y_k\})\subset G(n, \{x_1,..., x_r\}).$$

Therefore, $\mathcal{F}(X)$ is a $\gamma$-space.
\end{proof}

A function $d: X\times X\rightarrow \mathbb{R}^{+}$ is a {\it symmetric} on the set $X$ if, for any $x, y\in X$, we have (1) $d(x, y)=0\Leftrightarrow x=y$ and (2) $d(x, y)=d(y, x)$. For each $x\in X$ and $\varepsilon>0$, put $B(x, \varepsilon)=\{y\in X: d(x, y)<\varepsilon\}$. A space $X$ is said to be {\it symmetrizable} if there exists a symmetric $d$ on $X$ such that $U\subset X$ is open iff for each $x\in U$ there is $\varepsilon>0$ such that $B(x, \varepsilon)\subset U$. A space $X$ is called to be {\it semi-metrizable} if there exists a symmetric $d$ on $X$ such that for every $x\in X$, the family $\{B(x, \varepsilon): \varepsilon>0\}$ forms a neighborhood base at $x$. Clearly, each semi-metrizable space is symmetrizable.

\begin{theorem}
Let $X$ be a topological space with a countable pseudocharacter.
Then the following statements are equivalent.

\begin{enumerate}
\item $X$ is a semi-metrizable spaces;
\item $\mathcal{F}(X)$ is symmetrzable;
\item $\mathcal{F}(X)$ is semi-metrizable.
\end{enumerate}
\end{theorem}

\begin{proof}

The implication of (3) $\to$ (2) is trivial. It suffices to prove
(1) $\to$ (3) and (2) $\to$ (1).

(1) $\to$ (3). Assume $X$ is a semi-metrizable space. Note that a
space is semi-metrizable if and only if  it is a first-countable,
semi-stratifiable space by \cite[Theorem 9.8]{G1984}. Then it follow
from \cite[Theorem 39]{PS2017} and \cite[Proposition 4.5]{M1951}
that $\mathcal{F}(X)$ is a fist-countable semi-stratifiable space,
that is, a semi-metrizable space.

(2) $\to$ (1). Since $X$ is a closed subspace of $\mathcal{F}(X)$,
the subspace $X$ is symmetrizable.  Obviously, $\mathcal{F}(X)$
contains no closed copy of $S_\omega$, then it follows from
Proposition ~\ref{p2} that $X$ contains no closed copy of $S_2$.
Then $X$ is Fr\'echet-Urysohn since  a sequential space in which
each singleton is a $G_\delta$-set is Fr\'echet-Urysohn if it
contains no closed copy of $S_2$, hence $X$ is semi-metric by
\cite[Theorem 9.6]{G1984}.
\end{proof}

Finally, we prove that for a $g$-first-countable space $X$ if $\mathcal{F}(X)$ is sequential then $\mathcal{F}(X)$ is
$g$-first-countable.

\begin{theorem}
Let $X$ be a $g$-first-countable space. Then $\mathcal{F}(X)$ is
$g$-first-countable if and only if $\mathcal{F}(X)$ is sequential.
\end{theorem}

\begin{proof}
Clearly, it suffices to consider the sufficiency. Assume that $\mathcal{F}(X)$ is sequential, then it suffices to prove that $\mathcal{F}(X)$ is
$snf$-countable. Take an arbitrary point $\{x_{1}, \ldots, x_{m}\}\in\mathcal{F}(X)$. Since $X$ is a $g$-first-countable space, $X$ is $snf$-countable, hence for each $i\leq m$ there exists a sequence of decreasing subsets $\mathcal{P}(i)=\{P_{n}(i): n\in\mathbb{N}\}$ of $X$ such that $\mathcal{P}(i)$ is an $sn$-network at $x_{i}$ in $X$. Let
$$\mathcal{P}=\{\langle P_{n}(1),..., P_{n}(m)\rangle: P_{n}(i)\in\mathcal{P}(i), n\in\mathbb{N}, i\leq m\}.$$
Obviously, it is easy to see that $\mathcal{P}$ is a decreasing network at $\{x_{1}, \ldots, x_{m}\}$ in $\mathcal{F}(X)$. We claim that $\mathcal{P}$ is an $sn$-network at $\{x_{1}, \ldots, x_{m}\}$. Suppose not, we assume that there exists a sequence $\{F_{k}\}_{k\in\mathbb{N}}$ converges to $\{x_{1}, \ldots, x_{m}\}$ in $\mathcal{F}(X)$ such that for each $P\in\mathcal{P}$ the sequence $\{F_{k}\}_{k\in\mathbb{N}}$ is not eventually in $P$. Without loss of generality, we may assume that for each $n, i\in\mathbb{N}$ we have $P_{n}(i)\cap P_{m}(i)=\emptyset$ for any $m\neq n$ and $x_{i}\in P_{n}(i)$. Moreover, without loss of generality, we may suppose that either there exists $1\leq i\leq m$ such that $F_{k}\cap P_{n}(i)=\emptyset$ for each $k, n\in\mathbb{N}$, or $F_{k}\setminus\bigcup_{i=1}^{m}P_{n}(i)\neq\emptyset$ for each $k, n\in\mathbb{N}$. In order to obtain a contradiction, we divide the proof into the following two cases.

\smallskip
{\bf Case 1} There exists $1\leq i\leq m$ such that $F_{k}\cap P_{n}(i)=\emptyset$ for each $k, n\in\mathbb{N}$.

\smallskip
Then the point $x_{i}$ is not an isolated point in $X$ since $\{F_{k}\}_{k\in\mathbb{N}}$ converges to $\{x_{1}, \ldots, x_{m}\}$. Put $K=\{x_{1}, \ldots, x_{m}\}\cup\bigcup_{k\in\mathbb{N}}F_{k}$. Then $K$ is compact in $X$ by \cite[Theorem2.5.2]{M1951}, hence from the sequentiality of $X$ there exists a sequence $\{a_{k}\}_{k\in\mathbb{N}}$ of $K\setminus\{x_{1}, \ldots, x_{m}\}$ such that $a_{k}\rightarrow x_{i}$ as $k\rightarrow\infty$. Since each $P_{n}(i)$ is a sequential neighborhood of $x_{i}$, there exists $l\in\mathbb{N}$ such that $a_{k}\in P_{n}(i)$ for any $k>l$; however, this is a contradiction with $F_{k}\cap P_{n}(i)=\emptyset$ for each $k\in\mathbb{N}$.

\smallskip
{\bf Case 2} For each $k, n\in\mathbb{N}$, $F_{k}\setminus\bigcup_{i=1}^{m}P_{n}(i)\neq\emptyset$.

\smallskip

 Fix an arbitrary $n\in\mathbb{N}$. For each $k\in\mathbb{N}$, take an arbitrary point $b_{k}\in F_{k}\setminus\bigcup_{i=1}^{m}P_{n}(i)$. Put $B=\{b_{k}: k\in\mathbb{N}\}$. If $B$ is finite, then there exist $b\in B$ and an increasing sequence $\{n_{j}\}_{j\in\mathbb{N}}$ such that $b_{n_{j}}=b$ for each $j\in\mathbb{N}$. Hence, for each $l\leq m$, we can take an open neighborhood $V(l)$ of $x_{l}$ in $X$ such that $b\not\in V(l)$, then $\langle
V(1),..., V(m)\rangle$ is an open neighborhood of $\{x_{1}, \ldots, x_{m}\}$in $\mathcal{F}(X)$; however, $F_{n_{j}}\not\in \langle
V(1),..., V(m)\rangle$ for each $j\in\mathbb{N}$, which is a contradiction. Now it suffices to consider that $B$ is an infinite set, then $B$ has a cluster point $c$ in $X$ since $K=\{x_{1}, \ldots, x_{m}\}\cup\bigcup_{k\in\mathbb{N}}F_{k}$ is compact in $X$ by \cite[Theorem 2.5.2]{M1951}. From the sequentiality of $X$, there exists a sequence $\{c_{k}\}_{k\in\mathbb{N}}$ of $B$ such that $\{c_{k}\}_{k\in\mathbb{N}}$ converges to $c$ in $X$. Without loss of generality, we may assume that $c_{k}\in F_{k}$ for each $k\in\mathbb{N}$. Clearly, $c\not\in\{x_{1}, \ldots, x_{m}\}$,
hence we can take an open neighborhoods $V(c)$ of $c$ and $V(k)$ of $x_{k}$ in $X$ for each $k\leq m$ such that $V(c)\cap V(k)=\emptyset$, then $\langle
V(1),..., V(m)\rangle$ is an open neighborhood of $\{x_{1}, \ldots, x_{m}\}$ in $\mathcal{F}(X)$, and there exists $L\in\mathbb{N}$ such that $c_{k}\in V(c)$ for any $k>L$ since $V(c)\cap V(j)=\emptyset$ for any $j\leq m$; therefore, $F_{k}\not\in\langle
V(1),..., V(m)\rangle$ for any $k>L$, which is a contradiction.
\end{proof}

However, the following two questions are interesting and remain open.

\begin{question}
Let $X$ be a $g$-first-countable space. If $X^n$ is sequential for
$n\in \mathbb{N}$, is $\mathcal{F}(X)$ g-first-countable?
\end{question}

\begin{question}
Is $\mathcal{F}(S_2)$ $g$-first-countable?
\end{question}


\begin{thebibliography}{99}
\bibitem{AJ2000} B. Alleche, A.V. Arhangel'ski\v{\i}, J. Calbrix, Weak developments and metrization,
Topol. Appl., 100(2000), 23--38.

\bibitem{BG2005} B. Bailey, G. Gruenhage, On a question concerning sharp bases, Topol. Appl., 153(2005), 90--96.

\bibitem{BBK2008} T.O. Banakh, V.I. Bogachev, A.V. Kolesnikov, $k^*$-metrizable spaces and their applications, J. Math. Sci., 155(2008), 475--522.

 \bibitem{B2016}  T. Banakh, Fans and their applications in General Topology, Functional Analysis and Topological Algebra, arXiv:1602.04857v2.

\bibitem{B1980} D. Burke, Closed mappings, in: G. Reed ed., Surveys
in General Topology, Academic Press, New York, 1980, 1--32.

\bibitem{BBL2006} H. Bennet, R. Byerly, D. Lutzer, Compact
$G_\delta$ sets, Topol. Appl., 153(2006), 2169--2181.

\bibitem{BL1972} H. Bennett, D. Lutzer, A note on weak
$\theta$-refinability, Gen. Topol. Appl., 1(1972), 253--262.

\bibitem{E1989} R. Engelking, {\it General topology} (revised and completed edition), Heldermann Verlag, Berlin, 1989.

\bibitem{G1984} G. Gruenhage, {\it Generalized metric spaces}, In: K. Kunen, J. E.
  Vaughan(Eds.), Handbook of set-theoretic topology, North-Holland, Amsterdam, 1984, 423--501.

\bibitem{GM2016} C. Good, S. Macias, Symmetric products of
generalizedd metric spaces, Topol. Appl.,
206(2016), 93--114.

\bibitem{GMT1984} G. Gruenhage, E. Michael, Y. Tanaka, Spaces
dtermined by point-countable covers, Pacific J. Math., 113(1984),
303--332.

\bibitem{H1971} R. Hodel, Moore spaces and $w\Delta$-spaces, Pacific
J. Math., 38(1971), 641--652.

\bibitem {JY1992} H. Junnila, Z. Yun, $\aleph$-spaces and spaces
with $\sigma$-hereditarily closure-preserving $k$-network, Topol.
Appl., 44(1992), 209--215.

\bibitem {K1978} M.Kubo, A note on hyperspaces by compact sets,
Memoirs of Osaka Kyoiku Univ., Ser. III, 27(1978), 81--85.

\bibitem{LT1998} C. Liu, Y. Tanaka, Star-countable k-networks, and quotient images of locally
separable metric spaces, Topol. Appl., 82(1998),
317--325.

\bibitem{M1973} H. Martin, Metrizability of $M$-spaces, Canad. J.
Math., 25(1973), 840--841.

\bibitem{M1951} E. Michael, Topologies on spaces of subsets, Trans.
Amer. Math. Soc., 71(1951), 91--138.

\bibitem{LH2005} L. Mou and H. Ohta, Sharp bases and mappings, Houston J. Math,
31(2005), 227--238.

\bibitem{NP2015} S.A. Naimpally, J.F. Peters, Hyperspace Topologies, Topology with Applications: Topological Spaces via Near and Far, World Scientific, Singapore, 2013.

\bibitem{N1985} I. Ntantu, Cardinal functions on hyperspaces and
function spaces, Topol. Proc., 10(1985), 357--375.

\bibitem{PS2017} L. Peng, Y. Sun, A study on symmetric products of
generalized metric spaces, Topol. Appl., 231(2017), 411--429.

\bibitem{TLL2018} Z.B. Tang, S. Lin, F. Lin, Symmetric products and closed finite-to-one mappings, Topol. Appl., 234(2018),
26--45.

\bibitem{WW1965} H. Wicke, J. Worrell, Characterizations of
developable topological spaces, Canad. J. Math., 17(1965), 820--830.
\end{thebibliography}
\end{document}